\documentclass{amsart}
\usepackage{amsmath,amsfonts,amsthm,amssymb,verbatim,enumerate,graphicx}
\usepackage[flushmargin]{footmisc}
\newtheorem{theorem}{Theorem}[section]

\newtheorem{lemma}{Lemma}[section]
\newtheorem{corollary}{Corollary}[section]

\newtheorem{remark}{Remark}
\numberwithin{equation}{section}
\def \isnatural {\in\mathbb{N}}

\def \sai  {annular itinerary}
\def \sais {annular itineraries}

\newcommand{\tef}{transcendental entire function}
\newcommand\qfor{\quad\text{for }}
\newcommand\Real{\operatorname{Re}}
\newcommand\Imag{\operatorname{Im}}
\newcommand\Lunif{L_U}
\newcommand\Lflat{L_A}
\newcommand\Mflat{M_A}
\newcommand\thefun{E_\lambda}
\makeatletter
\def\blfootnote{\xdef\@thefnmark{}\@footnotetext}
\makeatother
\begin{document}
%
%
\title[Slowly escaping sets and annular itineraries]{Dimensions of slowly escaping sets and annular itineraries for exponential functions}
\author{D. J. Sixsmith}
\address{Department of Mathematics and Statistics \\
	 The Open University \\
   Walton Hall\\
   Milton Keynes MK7 6AA\\
   UK}
\email{david.sixsmith@open.ac.uk}
%
%
\begin{abstract}
We study the iteration of functions in the exponential family. We construct a number of sets, consisting of points which escape to infinity `slowly', and which have Hausdorff dimension equal to $1$. We prove these results by using the idea of an \emph{annular itinerary}. In the case of a general {\tef} we show that one of these sets, the \emph{uniformly slowly escaping set}, has strong dynamical properties and we give a necessary and sufficient condition for this set to be non-empty.
\end{abstract}
\maketitle
%
%
%
\blfootnote{The author was supported by Engineering and Physical Sciences Research Council grant EP/J022160/1.}
\section{Introduction}
This paper is principally concerned with {\tef}s in the \emph{exponential family}, defined by $$\thefun(z) = \lambda e^z, \qfor \lambda \in \mathbb{C}\backslash\{0\}.$$

For a general {\tef} $f$, the \emph{Fatou set} $F(f)$ is defined as the set $z~\in~\mathbb{C}$ such that $\{f^n\}_{n\isnatural}$ is a normal family in a neighbourhood of $z$. The \emph{Julia set} $J(f)$ is the complement in $\mathbb{C}$ of $F(f)$. An introduction to the properties of these sets was given in \cite{MR1216719}. The \emph{escaping set}, which was first studied for a general {\tef} in \cite{MR1102727}, is defined by $$I(f) = \{z : f^n(z)\rightarrow\infty\text{ as }n\rightarrow\infty\}.$$

Many authors have studied the Hausdorff dimension of $I(\thefun)$, or subsets of this set. We refer to \cite{falconer} for a definition of Hausdorff dimension, which we denote here by $\dim_H$. A key result is that of McMullen \cite{MR871679}, who showed that $\dim_H J(\thefun) = 2$. It is well-known that it follows from his construction that $\dim_H I(\thefun) = 2$.

When $\lambda \in (0, e^{-1})$ it is known -- see \cite{MR758892, MR873428} -- that $J(\thefun)$ consists of an uncountable set of unbounded curves known as a \emph{Cantor bouquet}, and that each curve, except possibly for its finite endpoint, lies in $I(\thefun)$. In two celebrated papers \cite{MR1680622, MR1696203} Karpi{\'n}ska proved the paradoxical fact that the set consisting of these curves excluding their finite endpoints has Hausdorff dimension $1$, whereas the set of finite endpoints has Hausdorff dimension $2$. A somewhat related result is that of Karpi{\'n}ska and Urba{\'n}ski \cite{MR2197375} who defined subsets of $I(\thefun)$ of Hausdorff dimension $d$, for each $d \in (1,2)$.

In fact, all these papers considered a subset $A(\thefun)$ of $I(\thefun)$ known as the \emph{fast escaping set}. The fast escaping set was introduced in \cite{MR1684251}, and can be defined \cite{Rippon01102012} for a general {\tef} $f$ by
\begin{equation}
\label{Adef}
A(f) = \{z : \text{there exists } \ell \isnatural \text{ such that } |f^{n+\ell}(z)| \geq M^n(R,f), \text{ for } n \isnatural\}.
\end{equation}
Here the \emph{maximum modulus function} is defined by $M(r,f) = \max_{|z|=r} |f(z)|,$ for $r \geq 0.$ We write $M^n(r,f)$ to denote repeated iteration of $M(r,f)$ with respect to the variable $r$. In (\ref{Adef}), $R > 0$ is such that $M^n(R,f)\rightarrow\infty$ as $n\rightarrow\infty$.

It is well-known that the sets constructed in \cite{MR1680622, MR1696203} and \cite{MR871679} lie in $A(\thefun)$. We show in Section~\ref{SKU} that this is also the case for the sets defined in \cite{MR2197375}.

It seems that little is known about the dimension of subsets of $I(\thefun)\backslash A(\thefun)$. Indeed, very little is known about the dimension of  $J(f) \cap(I(f)\backslash A(f))$ for any {\tef} $f$, with three notable exceptions. Bishop \cite{Bish2} constructed a {\tef} $f_1$ such that $\dim_H I(f_1)\backslash A(f_1) = 0$. At the other extreme, Eremenko and Lyubich \cite[Example 4]{MR918638} constructed a {\tef} $f_2$ such that $J(f_2) \cap(I(f_2)\backslash A(f_2))$ has positive area.

The remaining exception concerns the \emph{Eremenko-Lyubich class}, $\mathcal{B}$, which is defined as the class of {\tef}s for which the set of singular values is bounded. Clearly $\thefun\in\mathcal{B}$, for $\lambda\ne 0$. If $f \in\mathcal{B}$, then $I(f) \subset J(f)$ \cite[Theorem 1]{MR1196102} and so 
\begin{equation}
\label{setthing}
J(f) \cap(I(f)\backslash A(f)) = I(f)\backslash A(f), \qfor f \in\mathcal{B}. 
\end{equation}

Bergweiler and Peter \cite{MR3054344} studied the dimension of subsets of $I(f)$ consisting of points for which there is a completely general upper bound on the rate of escape. The following is part of \cite[Theorem 1]{MR3054344}.
\begin{theorem}
\label{TBP}
Suppose that $f \in \mathcal{B}$ and that $(p_n)_{n\geq 0}$ is a sequence of real numbers tending to infinity. Define $$\operatorname{Esc}(f,(p_n)) = \{ z \in I(f) : \text{ there exists } N \geq 0 \text{ such that } |f^n(z)| \leq p_n, \text{ for } n \geq N\}.$$ Then $\dim_H \operatorname{Esc}(f,(p_n)) \geq 1.$
\end{theorem}
Suppose that $f\in\mathcal{B}$. In contrast to Bishop's result, it follows from Theorem~\ref{TBP} and (\ref{setthing}) that $\dim_H J(f) \cap (I(f)\backslash A(f)) \geq 1$. Rempe and Stallard \cite{MR2587450} showed that there exists a {\tef} $f_3 \in \mathcal{B}$ such that $\dim_H I(f_3) = 1$. It follows from Theorem~\ref{TBP} and (\ref{setthing}) that $\dim_H J(f_3) \cap (I(f_3)\backslash A(f_3)) = 1$. \\ 

In this paper we show that various subsets of $J(\thefun) \cap (I(\thefun)\backslash A(\thefun))$ have Hausdorff dimension exactly equal to $1$. We do not use Theorem~\ref{TBP}, since all our results give an exact value for the Hausdorff dimension of sets defined by a two-sided inequality on the rate of escape. However, it seems plausible that there is some relationship between these results.

For a general {\tef} $f$ we define the \emph{uniformly slowly escaping set} by
\begin{equation}
\label{LUdef}
\Lunif(f) = \{z : \exists N\isnatural, \ R > 1, \ 0 < C_1 < C_2 \text{ s.t. } C_1R^{n} \leq |f^{n}(z)| \leq C_2R^{n}, \text{ for } n\geq N \}.
\end{equation}
Roughly speaking, this set consists of those points for which the rate of escape is eventually uniformly slow. Our first result concerns the Hausdorff dimension of $\Lunif(\thefun)$.
\begin{theorem}
\label{Tmain}
Suppose that $\lambda\ne 0$. Then $\dim_H \Lunif(\thefun) = 1$.
\end{theorem}
In Section~\ref{Sese} we give, for a general {\tef}, a necessary and sufficient condition for the uniformly slowly escaping set to be non-empty. We also prove that when the uniformly slowly escaping set is not empty, it has a number of familiar properties which show that, in general, this is a dynamically interesting set.

For a general {\tef} $f$, $\Lunif(f)$ is a subset of the \emph{slow escaping set}, introduced by Rippon and Stallard \cite{MR2792984}, and defined by
\begin{equation}
\label{Ldef}
L(f) = \{z \in I(f) : \text{ there exists } R > 1 \text{ s.t. } |f^{n}(z)| \leq R^{n}, \text{ for } n\isnatural \}.
\end{equation}
It was shown in  \cite{MR2792984} that $L(f)\ne\emptyset$, that $J(f)$ is dense in $L(f)$ and also that $J(f)=\partial L(f)$. 

It follows from Theorem~\ref{TBP} that $\dim_H L(\thefun) \geq 1$. Nothing more seems to be known about the actual dimension of $L(\thefun)$. As a step in that direction, we consider, for a general {\tef} $f$, a set which is a relatively large subset of $L(f)$ and which contains $\Lunif(f)$. First, for $p\isnatural$, let $\log^{+p}$ denote $p$ iterations of the $\log^+$ function, which is defined by
\begin{equation*}
 \log^+ (x) =
  \begin{cases}
   \log x, &\text{if } x \geq 1, \\
   0,  &\text{ otherwise}.
  \end{cases}
\end{equation*}

For a general {\tef}, $f$, we define
\begin{equation}
\label{Lflatdef}
\Lflat(f) = \{z : \exists R > 1, \ N, p\isnatural \text{ s.t. } n^{\log^{+p}(n)} \leq |f^{n}(z)| \leq R^n, \text{ for } n\geq N \}.
\end{equation}
Note that the orbits of points in $\Lflat(f)$ are constrained to lie within certain annuli. The fact that $$\Lunif(f) \subset \Lflat(f)\subset L(f) \subset I(f)\backslash A(f)$$ follows from (\ref{LUdef}), (\ref{Ldef}), (\ref{Lflatdef}) and well-known properties of the maximum modulus function. Our result concerning the dimension of $\Lflat(\thefun)$ is as follows.
\begin{theorem}
\label{Tthird}
Suppose that $\lambda\ne 0$. Then $\dim_H \Lflat(\thefun) = 1$.
\end{theorem}
Theorem~\ref{Tthird} is a consequence of the size of the annuli in the definition of $\Lflat(\thefun)$. In particular, Theorem~\ref{Tthird} shows that if $\dim_H L(\thefun) > 1$, then the vast majority of points in $L(\thefun)$ must have an extremely slowly escaping subsequence.

The techniques that we use to prove Theorem~\ref{Tthird} also allow us to construct subsets of $I(\thefun)\backslash(L(\thefun)\cup A(\thefun))$ which have Hausdorff dimension equal to $1$. For example, for a {\tef} $f$, Rippon and Stallard \cite{MR2792984} defined the \emph{moderately slow escaping set} by $$M(f) = \{z \in I(f) : \text{there exists } C > 0 \text{ such that } |f^{n}(z)| \leq \exp(e^{Cn}), \text{ for } n\isnatural \}.$$ In a similar way to (\ref{Lflatdef}), we define a subset of $M(f)\backslash L(f)$ by
\begin{equation}\label{Mflatdef}\Mflat(f) = \{z : \exists N, p\isnatural \text{ s.t. } e^{n\log^{+p}(n)} \leq |f^{n}(z)| \leq \exp(e^{pn}), \text{ for } n\geq N \}.\end{equation} Our result concerning the dimension of $\Mflat(\thefun)$ is as follows.
\begin{theorem}
\label{Tfourth}
Suppose that $\lambda\ne 0$. Then $\dim_H \Mflat(\thefun) = 1$.
\end{theorem}

We prove our results using the idea of an {\sai}. Before defining this concept, we briefly discuss a different type of itinerary which has frequently been used to study the dynamics of functions in the exponential family.

Since $|\thefun(z)| = |\lambda| e^{\Real(z)}$, it follows that the orbit of a point in $I(\thefun)$ must eventually remain in the right half-plane $\mathbb{H} = \{z : \Real(z) > 0\}$. Many authors -- see, for example, \cite{MR1721835,MR758892} and \cite{MR2197375} -- have considered itineraries of points in $I(\thefun)$ defined in the following way. First we partition $\mathbb{H}$ into half-open strips
\begin{equation}
\label{Vdef}
V_n = \{ z \in \mathbb{H} : (2n-1)\pi \leq \Imag(z) < (2n+1)\pi, \text{ for } n\in\mathbb{Z}\}.
\end{equation}
Suppose that $\underline{s} = s_0 s_1 s_2 \ldots$ is a sequence of integers. We say that a point $z$ has \emph{itinerary} $\underline{s} = \underline{s}(z) = s_0 s_1 s_2 \ldots$ if $\thefun^n(z) \in V_{s_n}$, for $n\geq 0$. 

For some types of itinerary it can be shown that the set of points with such an itinerary is -- in some sense -- large. For example, it follows from McMullen's proof \cite{MR871679} (and see also \cite{MR1680622}) that the set $$\{ z \in I(\thefun) : 2\pi|s_n(z)| \geq |\thefun^n(z)|/2, \text{ for } n\geq 0\}$$ has Hausdorff dimension $2$.

The concept of an \emph{annular itinerary} was introduced by Rippon and Stallard \cite{2013arXiv1301.1328R}. Suppose that $f$ is a general {\tef}, and let $(R_n)_{n\geq 0}$ be a strictly increasing sequence of positive real numbers such that $R_n\rightarrow\infty$ as $n\rightarrow\infty$. The strips $V_n$ in (\ref{Vdef}) are replaced by half-open annuli
\begin{equation*}
A_n = \{ z : R_{n-1} \leq |z| < R_n\}, \qfor n\isnatural,
\end{equation*}
and $A_0$ is defined as $\{ z : |z| < R_0\}$. Suppose that $\underline{t} = t_0 t_1 t_2 \ldots$ is a sequence of non-negative integers. If $f^n(z) \in A_{t_n}$, for $n\geq 0$, then we say that the point $z$ has annular itinerary $\underline{t} = \underline{t}(z) = t_0 t_1 t_2 \ldots$ with respect to the partition $(A_n)_{n\geq 0}$.

Rippon and Stallard \cite{2013arXiv1301.1328R} let $R_0>0$ be sufficiently large that $M^n(R_0,f)\rightarrow\infty$ as $n\rightarrow\infty$, and then set $R_n = M^n(R_0,f)$, for $n\isnatural$. They showed that, with this choice of partition $(A_n)_{n\geq 0}$, there is a very broad class of annular itineraries such that the set of points with such an itinerary contains a point in $J(f)$. For more information regarding the properties of these annular itineraries, we refer to \cite{2013arXiv1301.1328R}.

Annular itineraries are a natural choice when studying points which escape to infinity with different rates. The {\sais} used in our paper are defined using annuli of constant modulus, which seems a natural choice when considering points in the slow escaping set. First we choose a value of $R > 1$, and then set $R_n = R^{n+1}$, for $n\geq 0$. This construction of the partition $(A_n)_{n\geq 0}$ should be considered to be in place throughout the remainder of this paper. Note that this construction depends on $R$. Here, and elsewhere, we suppress some dependencies for simplicity of notation, and retain only dependencies which need to remain explicit.

We use the following notation $$I_{R}(\underline{t}) = \{ z : z \text{ has {\sai} } \underline{t} \text{ with respect to the partition } (A_n)_{n\geq 0}\}.$$

We are interested in a particular type of {\sai}. We say that an {\sai} $\underline{t} = t_0 t_1 t_2 \ldots$ is \emph{non-zero} if $t_n \ne 0$, for $n\geq 0$, \emph{escaping} if $t_n\rightarrow\infty$ as $n\rightarrow\infty$, \emph{admissible} if
$e^{t_n} > t_{n+1}$, for $n\geq 0$,
and \emph{slowly-growing} if
\begin{equation}
\label{G3}
\lim_{n\rightarrow\infty} \frac{t_n}{\sum_{k=1}^{n-1} t_k} = 0.
\end{equation}

Our main result regarding {\sais} is as follows.
\begin{theorem}
\label{Tfirst}
Suppose that $\lambda\ne 0$, $R > 1$ and $\underline{t}$ is an escaping {\sai}. Then $\dim_H I_{R}(\underline{t}) \leq 1.$ Moreover, there exists $R_0 = R_0(\lambda) > 1$ such that if, in addition, $R \geq R_0$ and $\underline{t}$ is non-zero, admissible and slowly-growing, then $\dim_H I_{R}(\underline{t})~=~1.$
\end{theorem}
\begin{remark}\normalfont
It seems surprising that, for a large class of {\sais}, the sets of points with the same {\sai} all have the same Hausdorff dimension. We note that there are annular itineraries of arbitrarily slow growth which satisfy the conditions of Theorem~\ref{Tfirst}. In other words, if $(p_n)_{n\geq 0}$ is a sequence of positive integers such that $p_n\rightarrow\infty$ as $n\rightarrow\infty$, then there exists an annular itinerary $\underline{t} = t_0 t_1 t_2 \ldots$ and $R>1$ such that $\dim_H I_R(\underline{t}) = 1$ and $t_n \leq p_n$, for $n\geq 0$.
\end{remark}
\begin{remark}\normalfont
We comment briefly on the final two conditions in the second part of Theorem~\ref{Tfirst}. The condition that the annular itinerary be admissible is required to ensure that $I_{R}(\underline{t})$ is not empty.
It is unclear if the condition (\ref{G3}) is essential, though it is required for our method of proof. It is a straightforward calculation to show that a sequence $(t_n)_{n\isnatural}$ which satisfies this condition also satisfies
\begin{equation}
\label{growthcond}
\log t_n = o(n) \text{ as } n\rightarrow\infty.
\end{equation}
However, the condition (\ref{growthcond}) is weaker than the condition (\ref{G3}). For example, consider the sequence defined by
\begin{equation*}
t_n = 2^{m^2}, \qfor (m-1)^3 \leq n < m^3, \ m\isnatural.
\end{equation*}
It can be shown that this sequence satisfies (\ref{growthcond}) but not (\ref{G3}). The techniques of this paper do not allow us to replace (\ref{G3}) with the apparently simpler condition (\ref{growthcond}).
\end{remark}
Finally, we note that the dimension of subsets of $J(\thefun)$ which lie outside of $I(\thefun)$ was studied in \cite[Theorem 2]{MR1680622} and \cite{MR1992945}. In addition, Pawelec and Zdunik \cite{2014arXiv1405.7784P} recently showed that, for certain values of $\lambda$, there exist indecomposable continua in $J(\thefun)$ which are of Hausdorff dimension $1$. These continua intersect with the fast escaping set. We refer to \cite{2014arXiv1405.7784P} for further details. \\

%
%
The structure of this paper is as follows. First, in Section~\ref{Shann}, we give some preliminary lemmas. In Section~\ref{Slower} we prove a theorem which gives a lower bound on the Hausdorff dimension of $I_{R}(\underline{t})$ for a certain type of annular itinerary. In Section~\ref{Supper} we prove a theorem which gives an upper bound on the Hausdorff dimension of certain sets. All our dimension results are consequences of these two theorems. In Section~\ref{Ssgi} we prove Theorem~\ref{Tmain}, Theorem~\ref{Tthird}, Theorem~\ref{Tfourth} and Theorem~\ref{Tfirst}. In Section~\ref{Sese}, we state and prove two results about the uniformly slowly escaping set. Finally, in Section~\ref{SKU}, we discuss, briefly, the result of Karpi{\'n}ska and Urba{\'n}ski mentioned earlier.
%
%
%
%
%
%
%
\section{Preliminary lemmas}
\label{Shann}
We start this section with two lemmas concerning functions in the exponential family. We define closed annuli and half-annuli, for $0 < r_1 < r_2$, by
\begin{equation}
\label{Hdef}
A(r_1, r_2) = \{z : r_1 \leq |z| \leq r_2\}\text{ and }H(r_1, r_2) = \{z \in A(r_1,r_2) :  \operatorname{Re}(z) \geq 0 \}.
\end{equation}
For $r>0$ and $a\in\mathbb{C}$, we write $B(a, r)$ for the open disc $\{ z : |z - a| < r\}$.

The first lemma provides an estimate on the density of preimages of one half-annulus in another; see Figure~\ref{f1}. Here, for measurable sets $U$ and $V$, we define $$\operatorname{dens}(U, V) = \frac {\operatorname{area} (U \cap V)}{\operatorname{area}(V)}, $$ where area$(U)$ denotes the Lebesgue measure of $U$.
\begin{figure}[ht]
	\centering
	\includegraphics[width=12cm,height=9cm]{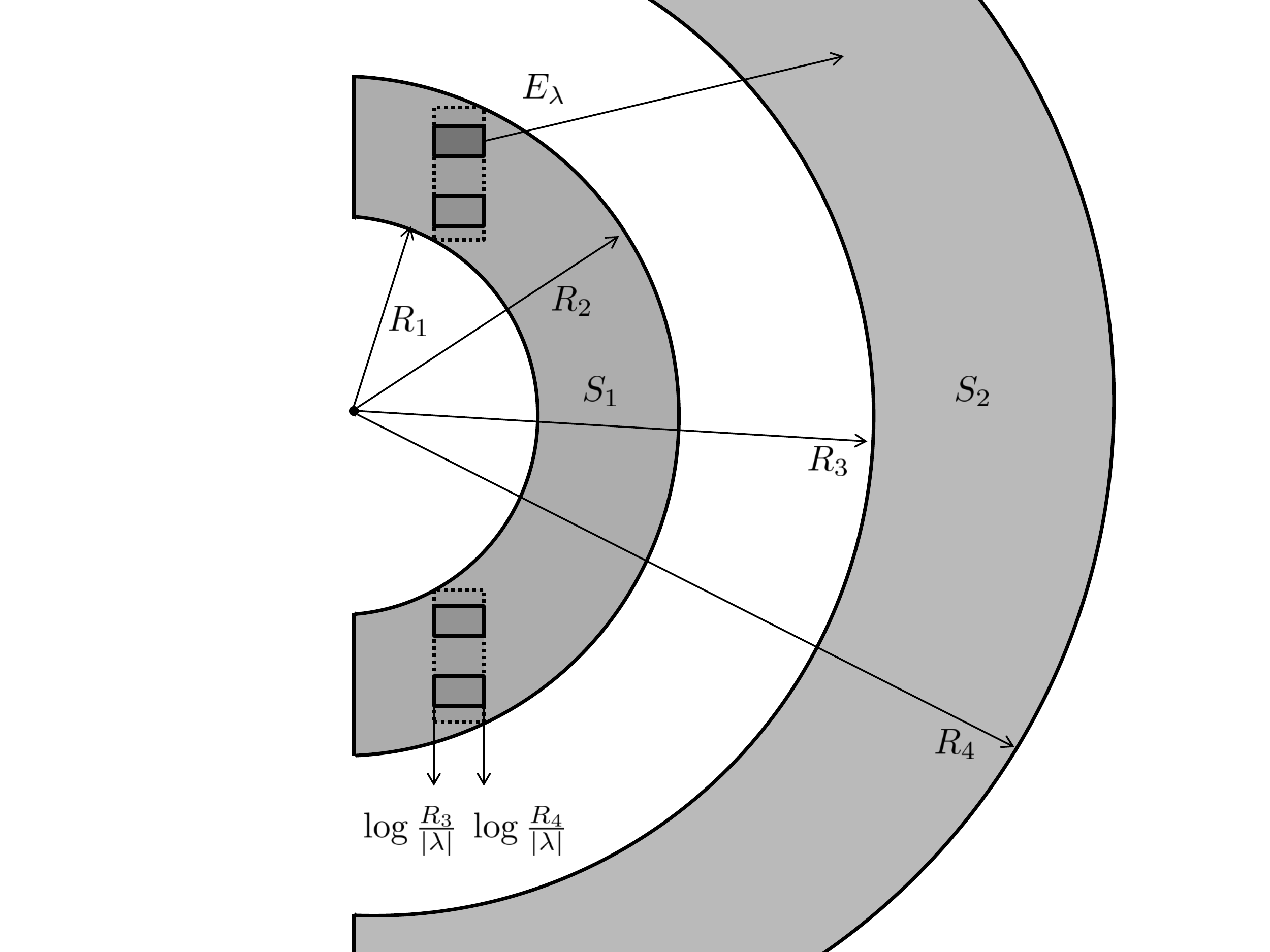}
	\caption{The set $\thefun^{-1}(S_2) \cap S_1$. One preimage component of $S_2$ is shown with a slightly darker background. The two rectangles constructed in the proof of Lemma~\ref{estlem} are shown with a dashed boundary. Note that $R_3$ is not necessarily larger than $R_2$.}
	\label{f1}
\end{figure}
\begin{lemma}
\label{estlem}
Suppose that $0 < R_1 < R_2$ and $0 < R_3 < R_4$ are such that
\begin{equation}
\label{est0}
R_2 > \max\left\{2 R_1, \ R_1 + 16\pi, \ 3 \log \frac{R_4}{|\lambda|}\right\},
\end{equation}
and
\begin{equation}
\label{est4}
R_3 > |\lambda|.
\end{equation}
Let $S_1 = {H(R_1, R_2)}$, $S_2 = {H(R_3, R_4)}$, and let $D$ be the union of all the components of $\thefun^{-1}(S_2)$ which are contained in $S_1$. Then
\begin{equation}
\label{denseq}
\operatorname{dens}(D, S_1) \geq \frac{1} {2\pi R_2}\log \frac{R_4}{R_3}.
\end{equation}
\end{lemma}
\begin{proof}
Each component of $\thefun^{-1}(S_2)$ is a rectangle of the form, for $n \in\mathbb{Z}$,
\begin{equation}
\label{rectdef}
\left\{ z : \log\frac{R_3}{|\lambda|} \leq \operatorname{Re}(z) \leq \log\frac{R_4}{|\lambda|}, \ \left(2n-\frac{1}{2}\right)\pi \leq \operatorname{Im}(z) + \arg(\lambda) \leq \left(2n+\frac{1}{2}\right)\pi \right\}.
\end{equation}
Suppose that the inequalities (\ref{est0}) and (\ref{est4}) both hold. Consider two large rectangles, each with sides parallel to the coordinate axes. One rectangle has a vertex at the point in the upper half-plane where the vertical line $\{ z : \operatorname{Re}(z) = \log\frac{R_3}{|\lambda|}\}$ meets the circle $B(0, R_1)$; note that if $R_1 \leq \log\frac{R_3}{|\lambda|}$ we put this vertex at $R_1$. The diagonally opposite vertex of this rectangle is at the point in the upper half-plane where the vertical line $\{ z : \operatorname{Re}(z) = \log\frac{R_4}{|\lambda|}\}$ meets the circle $B(0, R_2)$. The second rectangle is the  complex conjugate of the first one.

Let $h$ be the height of each rectangle. It follows by an application of Pythagoras's theorem to this rectangle, and by (\ref{est0}), that
\begin{align*}
h &= \left(R_2^2 - \left(\log \frac{R_4}{|\lambda|}\right)^2\right)^{\frac{1}{2}} - \left(\max\left\{0, \ R_1^2 - \left(\log \frac{R_3}{|\lambda| }\right)^2\right\}\right)^{\frac{1}{2}} \\
  &\geq \frac{7}{8}R_2 - R_1 \geq \frac{3}{4}(R_2 - R_1).
\end{align*}

It follows by (\ref{est0}) and (\ref{rectdef}) that each rectangle contains at least $\frac{1}{4\pi}(R_2 - R_1)$ components of $\thefun^{-1}(S_2)$. Hence $S_1$ contains at least $\frac{1}{2\pi}(R_2 - R_1)$ components of $\thefun^{-1}(S_2)$, each of which is a closed rectangle of height $\pi$ and width at least $\log \frac{R_4}{R_3}$. Equation (\ref{denseq}) follows from this, and the fact that area$(S_1) = \frac{1}{2}\pi(R_2^2 - R_1^2)$.
\end{proof}

%
%
For a domain $V$ and a {\tef} $f$, univalent in $V$, we define the \emph{distortion} of $f$ in $V$ by
\begin{equation}
D_V(f) = \frac{\sup_{z\in V}|f'(z)|}{\inf_{z\in V}|f'(z)|}.
\end{equation}
For functions in the exponential family, the following facts are immediate.
\begin{lemma}
Suppose that $F$ is a set such that $\inf\{|z| : z \in F \} = r_1 > 0$ and $\sup\{|z| : z \in F \} = r_2$, and that $V$ is a component of $\thefun^{-1}(F)$ such that $\thefun$ is univalent in $V$.  Then
\begin{equation}
\label{estlem2A}
|\thefun'(z)| \geq r_1, \qfor z \in V,
\end{equation}
and
\begin{equation}
\label{estlem2B}
D_V(\thefun) = \frac{r_2}{r_1}.
\end{equation}
\end{lemma}
We also use two well-known properties of Hausdorff dimension. For the first see, for example, \cite{falconer}.
\begin{lemma}
\label{Lcountstab}
Suppose that $(F)_{i\in I}$ is a collection of subsets of $\mathbb{C}$, and that $I$ is a finite or countable set. Then $$\dim_H \bigcup_{i\in I} F_i = \sup_{i\in I} \ \{\dim_H F_i\}.$$
\end{lemma}
The second property is used frequently but we are not aware of a reference.
\begin{lemma}
\label{Lsing}
Suppose that $f$ is a non-constant {\tef} and that $U \subset \mathbb{C}$. Then
\begin{equation*}
\dim_H f(U) = \dim_H f^{-1} (U) = \dim_H U.
\end{equation*}
\end{lemma}
\section{A lower bound on the Hausdorff dimension}
\label{Slower}
In this section we prove the following theorem which gives a lower bound on the Hausdorff dimension of $I_{R}(\underline{t})$ for a certain type of annular itinerary.
\begin{theorem}
\label{lbtheo}
Suppose that $\lambda\ne 0$. Then there exists $R_0 = R_0(\lambda) > 1$ such that, if $R \geq R_0$ and $\underline{t}$ is an escaping, non-zero, admissible and slowly-growing {\sai}, then $\dim_H I_{R}(\underline{t}) \geq 1.$
\end{theorem}
%
%
To prove Theorem~\ref{lbtheo}, we use a well-known construction and result of McMullen. Let $(\mathcal{E}_n)_{n\geq 0}$ be a sequence of finite collections of pairwise disjoint compact subsets of $\mathbb{C}$ such that the following both hold:
\begin{enumerate}[(i)]
\item If $F \in \mathcal{E}_{n+1}$, then there exists a unique $G \in \mathcal{E}_{n}$ such that $F \subset G$;
\item If $G \in \mathcal{E}_{n}$, then there exists at least one $F \in \mathcal{E}_{n+1}$ such that $G \supset F$.
\end{enumerate}
We write
\begin{equation}
\label{Edef}
D_n = \bigcup_{F\subset \mathcal{E}_n} F, \text{ for } n\geq 0, \quad \text{and} \quad D = \bigcap_{n\geq 0} D_n.
\end{equation}
McMullen's result is the following \cite[Proposition 2.2]{MR871679}. Here, for a set $U$, $\operatorname{diam} U$ denotes the Euclidean diameter of $U$.
\begin{lemma}
\label{mcmlemma}
Suppose that there exists a sequence of finite collections of pairwise disjoint compact sets, $(\mathcal{E}_n)_{n\geq 0}$, which satisfies conditions (i) and (ii) above, and let $D$ and $(D_n)_{n\geq 0}$ be as defined in (\ref{Edef}). Suppose also that $(\Delta_n)_{n\geq 0}$ and $(d_n)_{n\geq 0}$ are sequences of positive real numbers, with $d_n \rightarrow 0$ as $n\rightarrow\infty$, such that, for each $n\geq 0$ and for each $F \in \mathcal{E}_n$, we have $$\operatorname{dens}(D_{n+1}, F) \geq \Delta_n\quad\text{and}\quad\operatorname{diam} F \leq d_n.$$ Then
\begin{equation}
\label{mcmulleneq}
\dim_H D \geq 2 - \limsup_{n\rightarrow\infty} \frac{\sum_{m=0}^{n} | \log \Delta_m|}{| \log d_n|}.
\end{equation}
\end{lemma}
\begin{remark}
\normalfont In \cite{MR871679} the upper bound of summation in (\ref{mcmulleneq}) was given as $n+1$. However, the stronger result (\ref{mcmulleneq}) -- which is required in the proof of Theorem~\ref{lbtheo} -- follows from McMullen's proof, and has been given in, for example, \cite[Lemma 4.4]{MR2609307} and \cite[Lemma 4.3]{MR2458811}.
\end{remark}
%
%
We also use the following. This is a version of \cite[Lemma 5.2]{areapaper}, which itself is a detailed version of \cite[Proposition~3.1]{MR871679}.
\begin{lemma}
\label{distortion.lemma}
Suppose that $f$ is a {\tef}, and there exists a set $U\subset\mathbb{C}$ and constants $\alpha>1$ and $M > 0$ such that
\begin{equation}
\label{unifexp}
|f'(z)| > \alpha \quad\text{ and }\quad \left|\frac{f''(z)}{f'(z)}\right| < M, \qfor z \in U.
\end{equation}
Suppose also that there exists $s \in (0, (4M)^{-1})$ such that if $B \subset U$ is a disc of diameter $s$, then $f$ is conformal in a neighbourhood of $B$. Suppose finally that $(B_m)_{m\in\{1, 2, \ldots, n\}}$ is a sequence of sets contained in $U$, each of diameter less than $s$, and such that $$B_{m+1} \subset f(B_m), \qfor m\in\{1, 2, \ldots, n-1\}.$$ For $m\in\{1, 2, \ldots, n\}$, let $\phi_m$ be the inverse branch of $f$ which maps $f(B_m)$ to $B_m$, and set $V = \phi_1 \circ \phi_2 \circ \cdots \circ \phi_{n}(f(B_n))$. Then there exists $\tau=\tau(M, s, \alpha)>1$ such that $$D_V(f^n)\leq \tau.$$
\end{lemma}
Note that in \cite[Lemma 5.2]{areapaper} the sets $B_n$ are squares of side $s$. The proof of the above result follows in exactly the same way, and is omitted.\\

We deduce the following.
\begin{corollary}
\label{c1}
There exist absolute constants $s_0 > 0$ and $\tau_0 > 1$ such that the following holds. Suppose that $\lambda\ne 0$, $n\isnatural$ and $V$ is a set such that $$\thefun^m(V) \subset \left\{ z : \operatorname{Re}(z) > \log\frac{2}{|\lambda|} \right\} \text{ and } \operatorname{diam} \thefun^m(V) < s_0, \qfor 0 \leq m < n.$$ Then $D_V(\thefun^n)\leq \tau_0$.
\end{corollary}
\begin{proof}
This result follows from Lemma~\ref{distortion.lemma} with $f = \thefun$ , $U = \left\{ z : \operatorname{Re}(z) > \log\frac{2}{|\lambda|} \right\}$, $\alpha=M=2$, and $B_{m} = \thefun^{m-1}(V)$, for $1 \leq m \leq n$.
\end{proof}

%
%
We now give the proof of Theorem~\ref{lbtheo}. Roughly speaking, our method of proof is as follows. First we set a value of $R_0$ sufficiently large to enable us to use Lemma~\ref{estlem}. We then define a set which is contained in $I_{R}(\underline{t})$ and apply McMullen's result to obtain a lower bound on the Hausdorff dimension of this set. 
\begin{proof}[Proof of Theorem~\ref{lbtheo}]
Let $s_0$ be the constant in Corollary~\ref{c1}. We choose
\begin{equation}
\label{R0choice}
R_0>\max\left\{e, |\lambda|, \frac{2}{s_0} \right\},
\end{equation}
sufficiently large that
\begin{equation}
\label{Rineq}
R(R-1) > 16 \pi + 2\quad\text{ and }\quad R > 3 \log \frac{R^2}{|\lambda|}+1, \qfor R \geq R_0.
\end{equation}

Suppose that $R \geq R_0$, and that $\underline{t}$ is an escaping, non-zero, admissible and slowly-growing {\sai}. To use Lemma~\ref{mcmlemma} we need to work with compact and disjoint sets. 
In order to do this, and recalling the definition (\ref{Hdef}), we define disjoint closed half-annuli $$H_n = H(R^n+1, R^{n+1}-1), \qfor n\isnatural.$$

Since $\underline{t}$ is admissible and non-zero, we deduce by (\ref{R0choice}) and (\ref{Rineq}) that, for $n\geq 0$, we have
\begin{equation}
\label{sineq}
R^{t_n + 1} > R e^{t_n} > R t_{n+1} > 3 t_{n+1} \log R + 3 \log \frac{R}{|\lambda|} + 1 > 3 \log \left(\frac{R^{t_{n+1}+1}-1}{|\lambda|}\right) + 1.
\end{equation}
Since $\underline{t}$ is non-zero, we deduce from (\ref{R0choice}), (\ref{Rineq}) and (\ref{sineq}) that the hypotheses of Lemma~\ref{estlem} are satisfied with $S_1 = H_{t_n}$ and $S_2 = H_{t_{n+1}}$, for $n\geq 0$.

In order to use Lemma~\ref{mcmlemma}, we define a sequence of finite collections of pairwise disjoint compact sets as follows. First set $$\mathcal{E}_0 = \{ H_{t_0} \},$$ and, for $n\geq 0$, $$\mathcal{E}_{n+1} = \{ F : F \subset G, \text{ for some } G \in \mathcal{E}_n, \text{ and } \thefun^{n+1}(F) = H_{t_{n+1}}\}.$$

Let $D_n$, for $n\geq 0$, and $D$ be the sets defined in (\ref{Edef}). It follows from (\ref{Edef}) that $D \subset I_{R}(\underline{t}).$ It is sufficient, therefore, to show that $\dim_H D \geq 1$.

It follows from Lemma~\ref{estlem} that the conditions (i) and (ii) stated prior to Lemma~\ref{mcmlemma} are both satisfied. It remains to estimate the diameters and the densities stated in Lemma~\ref{mcmlemma}. Note that to apply equation (\ref{mcmulleneq}) we may omit the definition of a finite number of these estimates.

Since $\underline{t}$ is escaping, we can let $N_0\geq 2$ be sufficiently large that $t_{n-1} \geq 2$, for $n\geq N_0$. Suppose that $n\geq N_0$ and that $F \in \mathcal{E}_n$. Note that $\thefun^n(F) = H_{t_n}$. We first find an upper bound on the diameter of $F$. Since $$\text{diam } \thefun^n(F) = \text{diam }  H_{t_n} < 2R^{t_n+1},$$ we have, by (\ref{estlem2A}), that
\begin{equation}
\label{diambound}
\operatorname{diam } F \leq 2R^{t_n+1} \frac{1}{R^{t_1}} \frac{1}{R^{t_2}} \cdots \frac{1}{R^{t_n}} = 2R^{1 - \sum_{m=1}^{n-1} t_m}.
\end{equation}
We set $d_n = 2R^{1 - \sum_{m=1}^{n-1} t_m}$. Since $\underline{t}$ is escaping, we deduce that
\begin{equation}
\label{diams}
d_n \rightarrow 0 \quad\text{and}\quad |\log d_n| = \log R \ \sum_{m=1}^{n-1} t_m\left(1 + o(1)\right) \text{ as } n\rightarrow\infty.
\end{equation}

We next show that the distortion of $\thefun^n$ on $F$ is bounded independently of $n$ and $F$. Once again by (\ref{estlem2A}), and by (\ref{R0choice}), we have
\begin{equation*}
\operatorname{diam } \thefun^{m}(F) \leq 2R^{t_n+1} \frac{1}{R^{t_{n-1}}} \frac{1}{R^{t_n}}\leq \frac{2}{R} < s_0, \qfor 0 \leq m < n-1.
\end{equation*}

Suppose that $0 \leq m < n-1$ and that $z \in \thefun^m(F)$. Since $\underline{t}$ is non-zero, we have $|\thefun(z)| \geq R$. We deduce by (\ref{R0choice}) that $$\operatorname{Re}(z) \geq \log \frac{R}{|\lambda|} > \log \frac{2}{|\lambda|}.$$ Hence $$\thefun^m(F) \subset \left\{ z : \operatorname{Re}(z) > \log\frac{2}{|\lambda|} \right\}, \qfor 0 \leq m < n-1.$$

We deduce by Corollary~\ref{c1} that $D_F(\thefun^{n-1})\leq \tau_0$. Moreover, it follows from (\ref{estlem2B}) that $D_{\thefun^{n-1}(F)}(\thefun)\leq R$. Thus
\begin{equation}
\label{disteq}
D_F(\thefun^{n})\leq D_F(\thefun^{n-1}) \ D_{\thefun^{n-1}(F)}(\thefun) \leq \tau_0 R.
\end{equation}

We use (\ref{disteq}) to find a lower bound on dens$(D_{n+1}, F)$. Note that $\thefun^n(D_{n+1})$ consists of those components of $\thefun^{-1}(H_{t_{n+1}})$ which are contained in $H_{t_n}$. Hence, by (\ref{denseq}), (\ref{R0choice}) and (\ref{disteq}), we have
\begin{align*}
\operatorname{dens}(D_{n+1}, F) &\geq \frac{1}{(\tau_0 R)^2} \operatorname{dens}(\thefun^n(D_{n+1}), \thefun^n(F)) \\
                                &\geq \frac{1}{(\tau_0 R)^2} \frac{1}{2\pi\left(R^{t_n+1}-1\right)} \log \frac{R^{t_{n+1}+1}-1}{R^{t_{n+1}}+1} \\
                                &\geq \frac{\log R} {4\tau_0^2\pi R^{t_n+3}}.
\end{align*}
We set $\Delta_n = \frac{\log R} {4\tau_0^2\pi R^{t_n+3}}$. Note that
%
\begin{equation}
\label{dists}
|\log \Delta_n| = t_n\log R\left(1+o(1)\right) \text{ as } n\rightarrow\infty.
\end{equation}

Since $\underline{t}$ is escaping and slowly-growing, it follows, by (\ref{G3}), (\ref{diams}) and (\ref{dists}), that
\begin{align*}
\limsup_{n\rightarrow\infty} \frac{\sum_{m=0}^{n} |\log \Delta_m|}{|\log d_n|} 
   &= \limsup_{n\rightarrow\infty} \frac{\log R \ \sum_{m=1}^{n} t_m(1+o(1))}{\log R \ \sum_{m=1}^{n-1} t_m(1+o(1))} \\
   &= \limsup_{n\rightarrow\infty} \left(\frac{t_n(1+o(1))}{\sum_{m=1}^{n-1} t_m(1+o(1))} + \frac{\sum_{m=1}^{n-1} t_m(1+o(1))}{\sum_{m=1}^{n-1} t_m(1+o(1))}\right) \\
   &= 1.
\end{align*}
We deduce by Lemma~\ref{mcmlemma} that $\dim_H D \geq 1$, as required.
\end{proof}
%
%
%
%
\section{An upper bound on Hausdorff dimension}
\label{Supper}
In this section we prove a theorem which gives an upper bound on the Hausdorff dimension of certain sets, and so is, in a sense, complementary to Theorem~\ref{lbtheo}. First we define the sets. Suppose that, for each $p\isnatural$, $(g_{p,n})_{n\isnatural}$ and $(h_{p,n})_{n\isnatural}$ are sequences of positive real numbers such that
\begin{equation}
\label{bigenough}
h_{p,n} \geq g_{p,n}, \qfor n\isnatural.
\end{equation}
For $\lambda\ne 0$ define the set $T_{g,h}$ by
\begin{equation}
\label{Tdef}
T_{g,h} = \{z : \text{there exist } N, p\isnatural \text{ s.t. } \thefun^{n}(z) \in A(g_{p,n}, h_{p,n}), \text{ for } n \geq N \}.
\end{equation}
Our theorem is as follows. 
\begin{theorem}
\label{ubtheo}
Suppose that $\lambda\ne 0$, and that for each $p\isnatural$, $(g_{p,n})_{n\isnatural}$ and $(h_{p,n})_{n\isnatural}$ are sequences of positive real numbers such that (\ref{bigenough}) is satisfied,
\begin{equation}
\label{ghbasic}
g_{p,n} \rightarrow\infty \text{ as } n \rightarrow\infty,
\end{equation}
and
\begin{equation}
\label{gheq}
\frac{\log g_{p,n}}{\log^+ \log h_{p,n+1}} \rightarrow\infty \text{ as } n \rightarrow\infty.
\end{equation}
Then $\dim_H T_{g, h} \leq 1.$
\end{theorem}
\begin{proof}

%
Choose
\begin{equation}
\label{betadef}
\beta > \max\left\{\sqrt{1+\pi^2}, \ \frac{1}{s_0}\right\}
\end{equation}
and set
\begin{equation*}
c_0 = |\lambda|e^{1+\beta}\max\left\{\frac{2}{|\lambda|},\ \exp(\beta^2)\right\}
\end{equation*}
and
\begin{equation}
\label{c1def}
c_1 = \log\frac{c_0}{|\lambda|} = 1+\beta+ \max\left\{\log\frac{2}{|\lambda|},\  \beta^2\right\},
\end{equation}
where $s_0$ is the constant in Corollary~\ref{c1}. By (\ref{ghbasic}), for each $p\isnatural$ we can choose $\ell_p \isnatural$ such that
\begin{equation}
\label{geq}
g_{p,n} > c_0, \qfor n\geq \ell_p.
\end{equation}

Define sets
\begin{equation}
\label{Sdef}
S_{\nu, p} = \{ z : \thefun^n(z) \in A(g_{p,\nu+n}, h_{p,\nu+n}), \text{ for } n\geq 0\}, \qfor \nu, p \isnatural.
\end{equation}
Fix a value of $p\isnatural$, and suppose that $\nu\geq \ell_p$. We claim that
\begin{equation}
\label{Sclaim}
\dim_H S_{\nu,p} \leq 1.
\end{equation}
Before proving (\ref{Sclaim}) we show that the result of the lemma can be deduced from this equation. First we claim that
\begin{equation}
\label{Tclaim}
T_{g, h} \subset \bigcup_{p\isnatural} \bigcup_{\nu\geq \ell_p} \thefun^{-\nu}(S_{\nu,p}).
\end{equation}
For, suppose that $z \in T_{g,h}$, in which case there exist $N, p\isnatural$ such that $$\thefun^{n}(z)~\in~A(g_{p,n}, h_{p,n}), \qfor n \geq N.$$ It follows that there exists $\nu \geq \max\{N, \ell_p\}$ such that $\thefun^{n}(\thefun^\nu(z)) \in A(g_{p,n+\nu}, h_{p,n+\nu})$, for $n \geq 0$, and so $\thefun^\nu(z) \in S_{\nu,p}$. This establishes equation (\ref{Tclaim}). Theorem~\ref{ubtheo} follows from (\ref{Sclaim}) and (\ref{Tclaim}), by Lemma~\ref{Lcountstab} and Lemma~\ref{Lsing}. \\

If remains to show that (\ref{Sclaim}) holds for $\nu \geq \ell_p$. We establish this result using a sequence of covers of $S_{\nu, p}$ and basic properties of Hausdorff dimension. We suppress the variable $p$ for simplicity. 

First we note some properties of $S_{\nu}$ and then use these properties to define a sequence of covers of this set. It follows from (\ref{geq}), and the fact that $\nu \geq \ell$, that $$|\thefun(\thefun^n(z))| = |\lambda|e^{\operatorname{Re}(\thefun^n(z))} > c_0, \qfor z \in S_{\nu}, \ n\geq 0.$$ We deduce by (\ref{c1def}) that
\begin{equation}
\label{nicepoints}
\operatorname{Re}(\thefun^n(z)) > \log\frac{c_0}{|\lambda|} = c_1, \qfor z \in S_{\nu}, \ n\geq 0.
\end{equation}
Define compact sets
\begin{equation}
\label{Wmndef}
W_{n,m} = A(e^{m-1}g_{n}, e^{m}g_{n}) \cap \{ z : \operatorname{Re}(z)\geq c_1\}, \qfor n, m \isnatural.
\end{equation}
We observe that each component of $\thefun^{-1}(W_{n,m})$, for $m, n \isnatural$, is contained in a distinct rectangle of the form,
\begin{equation}
\label{rectangles}
\left\{ z : -1 \leq \operatorname{Re}(z) - m - \log g_n + \log |\lambda| \leq 0, \ -\pi/2 \leq \operatorname{Im}(z) + \arg(\lambda) + 2k\pi \leq \pi/2 \right\},
\end{equation}
for some $k \in\mathbb{Z}$. Hence, by (\ref{betadef}), if $F$ is a component of $\thefun^{-1}(W_{n,m})$, then
\begin{equation}
\label{thediam}
\operatorname{diam } F \leq \beta, \qfor m, n \isnatural.
\end{equation}
For simplicity of notation, define sets of integers
\begin{equation}
\label{cardalphan}
\alpha_n = \left\{ 1, 2, \ldots, \left[\log \frac{h_{n}}{g_{n}}\right]+1\right\}, \qfor n\isnatural,
\end{equation}
where $[x]$ denotes the integer part of $x$. Note that
\begin{equation}
\label{eqann}
A(g_{n}, h_{n}) \cap \{ z : \operatorname{Re}(z)\geq c_1\} \subset \bigcup_{m\in\alpha_n} W_{n,m}, \qfor n\isnatural.
\end{equation}
Now set $$\mathcal{E}_0 = \{ W_{\nu,m} : m \in \alpha_{\nu}\},$$ 
and, for $n\geq 0$, $$\mathcal{E}_{n+1} = \{F : \thefun^{n+1}(F) = W_{\nu+n+1,m} \text{ for some } m \in \alpha_{\nu+n+1} \text{ and } F \ \cap \ G \ne \emptyset \text{ for some } G \in \mathcal{E}_n\}.$$

It follows by (\ref{nicepoints}) and (\ref{eqann}) that $$S_{\nu} \subset \bigcup_{F \in \mathcal{E}_n} F, \qfor \text{each } n\geq 0.$$ This completes the definition of the sequence of covers of $S_{\nu}$.\\

Next we study the properties of the sets in these covers, particularly the diameters of these sets. Suppose that $q\isnatural$. We claim that the following hold for $n \geq q$ and $F \in \mathcal{E}_n$;
\begin{equation}
\label{nicediams}
\operatorname{diam } \thefun^{n-q}(F) \leq \beta^{3-2q},
\end{equation}
\begin{equation}
\label{nicerealparts0}
\thefun^{n-q}(F) \subset \left\{ z : \operatorname{Re}(z) \geq c_1 - \sum_{k=1}^q \beta^{3-2k}\right\},
\end{equation}
and
\begin{equation}
\label{nicerealparts}
\thefun^{n-q}(F) \subset \left\{ z : \operatorname{Re}(z) > \max\left\{\log\frac{2}{|\lambda|},\beta^2\right\}\right\}.
\end{equation}
First we note that (\ref{nicerealparts}) follows from (\ref{betadef}), (\ref{c1def}) and (\ref{nicerealparts0}).

We prove (\ref{nicediams}) and (\ref{nicerealparts0}) by induction on $q$. We consider first the case that $q=1$. Suppose that $n\isnatural$ and that $F \in \mathcal{E}_n$. By (\ref{thediam}), $\thefun^{n-1}(F)$ has diameter at most $\beta$. Moreover, since $$\thefun^{n-1}(F) \cap W_{\nu+n-1,m'} \ne \emptyset, \qfor \text{some } m' \in \alpha_{\nu+n-1},$$ we deduce by (\ref{Wmndef}) that $\thefun^{n-1}(F) \subset \{ z : \operatorname{Re}(z) \geq c_1 - \beta\}$. This establishes (\ref{nicediams}) and (\ref{nicerealparts0}) in the case that $q=1$.

Now, suppose that (\ref{nicediams}) and (\ref{nicerealparts0}) have been been established for $1 \leq q\leq s$, for some $s\isnatural$. Suppose that $n\geq s+1$, that $F \in \mathcal{E}_n$ and that $G \in \mathcal{E}_{n-1}$ is such that $F \cap G \ne \emptyset$. First, we deduce by (\ref{estlem2A}), and by (\ref{nicediams}) and (\ref{nicerealparts}) with $s$ in place of $q$, that
\begin{equation}
\label{nextdiam}
\operatorname{diam } \thefun^{n-(s+1)}(F) \leq \frac{\operatorname{diam } \thefun^{n-s}(F)}{\inf \{|z| : z \in \thefun^{n-s}(F)\}}\leq \frac{\beta^{3-2s}}{\beta^2}=\beta^{3-2(s+1)}.
\end{equation}
Second, applying (\ref{nicerealparts0}) with $G$ in place of $F$, $n-1$ in place of $n$, and $s$ in place of $q$, we deduce that
\begin{equation*}
\thefun^{(n-1)-s}(G) \subset \left\{ z : \operatorname{Re}(z) \geq c_1 - \sum_{k=1}^s \beta^{3-2k}\right\}.
\end{equation*}
Since $\thefun^{n-(s+1)}(F) \cap \thefun^{(n-1)-s}(G) \ne \emptyset$, it follows by (\ref{nextdiam}) that
\begin{equation*}
\thefun^{n-(s+1)}(F) \subset \left\{ z : \operatorname{Re}(z) \geq c_1 - \sum_{k=1}^s \beta^{3-2k} - \beta^{3-2(s+1)} \right\}.
\end{equation*}
By induction, this completes the proof of (\ref{nicediams}) and (\ref{nicerealparts0}). We observe that it follows from (\ref{nicediams}) that the diameters of the sets in $\mathcal{E}_n$ tend uniformly to zero as $n\rightarrow\infty$.\\

Suppose next that $\epsilon > 0$. By the definition of Hausdorff dimension, together with the observation above regarding the diameters of the sets in $\mathcal{E}_n$, our proof of (\ref{Sclaim}) is complete if we can show that $$\sum_{F\in\mathcal{E}_{n+1}} (\text{diam }F)^{1+\epsilon} \leq \sum_{F\in\mathcal{E}_{n}} (\text{diam }F)^{1+\epsilon}, \qfor \text{large }n,$$ or indeed if we can show that, for all sufficiently large $n$, we have for each $G \in \mathcal{E}_n$ that
\begin{equation}
\label{outcome}
\sum_{\substack{F\in\mathcal{E}_{n+1}, \\ F \cap G \ne \emptyset}} \left(\frac{\text{diam } F}{\text{diam } G}\right)^{1+\epsilon} \leq 1.
\end{equation}

Suppose that $n\isnatural$ and that $G \in \mathcal{E}_n$, in which case $\thefun^{n}(G) = W_{\nu+n,m}$, for some $m\in\alpha_{\nu+n}$. Suppose also that $F~\in~\mathcal{E}_{n+1}$ intersects with $G$, in which case $\thefun^{n}(F)$ intersects with $W_{\nu+n,m}$ and is a preimage component of $W_{\nu+n+1,m'}$, for some $m'\in\alpha_{\nu+n+1}$.

We note the following two simple estimates. First, it follows from (\ref{rectangles}) that for each $m'\in\alpha_{\nu+n+1}$, there are at most $O(e^{m}g_{\nu+n})$ preimage components of $W_{\nu+n+1,m'}$ which intersect with $W_{\nu+n,m}$. It follows from this estimate and from (\ref{cardalphan}), that the total number of elements of $\mathcal{E}_{n+1}$ which intersect with $G$ is at most
\begin{equation}
\label{numpreims}
O\left(e^{m}g_{\nu+n}\log\frac{h_{\nu+n+1}}{g_{\nu+n+1}}\right) \text{ as } n\rightarrow\infty.
\end{equation}
Second, it is immediate that
\begin{equation}
\label{diam2}
(\operatorname{diam } \thefun^{n}(G))^{-1} = O((e^{m}g_{\nu+n})^{-1}) \text{ as } n\rightarrow\infty.
\end{equation}

Next we estimate the distortion of $\thefun^{n}$ in $F \cup G$. Suppose that $n\geq 2$.
We deduce by (\ref{betadef}) and (\ref{nicediams}) that $$\operatorname{diam} \thefun^m(G) < s_0, \qfor 0 \leq m < n-1.$$

%
It follows by (\ref{nicerealparts}) and Corollary~\ref{c1} that $D_G(\thefun^{n-1})\leq \tau_0$. Moreover, it follows from (\ref{estlem2B}) that $D_{\thefun^{n-1}(G)}(\thefun)\leq e$. We deduce that $D_G(\thefun^{n}) = O(1) \text{ as } n\rightarrow\infty$. By a similar argument $D_F(\thefun^{n}) = O(1) \text{ as } n\rightarrow\infty$. Hence, since $F \cap G \ne \emptyset$, we have 
\begin{equation}
\label{disteq3}
D_{F \cup G} (\thefun^{n}) = O(1) \text{ as } n\rightarrow\infty.
\end{equation}

Combining these estimates, we deduce from (\ref{thediam}), (\ref{numpreims}),  (\ref{diam2}) and (\ref{disteq3}) that
\begin{align*}
\sum_{\substack{F\in\mathcal{E}_{n+1}, \\ F \cap G \ne \emptyset}} \left(\frac{\text{diam } F}{\text{diam } G}\right)^{1+\epsilon}
&\leq O(1) \sum_{\substack{F\in\mathcal{E}_{n+1}, \\ F \cap G \ne \emptyset}}\left(\frac{\text{diam } \thefun^{n}(F)}{\text{diam } \thefun^{n}(G)}\right)^{1+\epsilon} \\
&= O\left(e^{m}g_{\nu+n}\log\frac{h_{\nu+n+1}}{g_{\nu+n+1}} . (e^{m}g_{\nu+n})^{-(1+\epsilon)}\right) \\
&\leq O\left(g_{\nu+n}^{-\epsilon} \log h_{\nu+n+1}\right).
\end{align*}
We deduce by (\ref{gheq}) that (\ref{outcome}) holds, as required.
\end{proof}
\section{Dimension results}
\label{Ssgi}
In this section we prove Theorem~\ref{Tmain}, Theorem~\ref{Tthird}, Theorem~\ref{Tfourth} and Theorem~\ref{Tfirst}. In each case we use Theorem~\ref{lbtheo} to show that the Hausdorff dimension is bounded below by $1$, and we use Theorem~\ref{ubtheo} to show that the Hausdorff dimension is bounded above by $1$. It is somewhat simpler to prove Theorem~\ref{Tthird} before Theorem~\ref{Tmain}.
\begin{proof}[Proof of Theorem~\ref{Tthird}]
Choose $R \geq R_0$, where $R_0$ is as in the statement of Theorem~\ref{lbtheo}, and let $\underline{t}$ be the sequence 
\begin{equation}
\label{tdef}
\underline{t} = 1 \ 1 \ 1 \ 2 \ 3 \ \ldots \ (n-1) \ \ldots.
\end{equation}
We deduce by Theorem~\ref{lbtheo} that $\dim_H I_{R}(\underline{t}) \geq 1$. Suppose that $z \in I_{R}(\underline{t})$, in which case $R^{n-1} \leq |\thefun^n(z)| \leq R^{n}$, for $n\geq 2$. It follows that $z \in\Lflat(\thefun)$. Hence $I_{R}(\underline{t}) \subset \Lflat(\thefun)$, and so $\dim_H \Lflat(\thefun) \geq 1$. \\

Next, for each $p\isnatural$, let sequences $(g_{p,n})_{n\isnatural}$ and $(h_{p,n})_{n\isnatural}$ be defined by
\begin{equation}
\label{ghdef}g_{p,n}=n^{\log^{+p}(n) } \text{ and } h_{p,n} = e^{pn}, \qfor n\isnatural.
\end{equation}
We deduce by Theorem~\ref{ubtheo} that $\dim_H T_{g, h}\leq 1$. Suppose that $z \in \Lflat(\thefun)$, in which case, by (\ref{Lflatdef}), there exist $p, N\isnatural$ and $R>1$ such that $$n^{\log^{+p}(n)} \leq |\thefun^{n}(z)| \leq R^{n}, \qfor n\geq N.$$ If $p$ is sufficiently large, then $R^n \leq e^{pn}$, for $n\isnatural$. We deduce that $z \in T_{g, h}$, where $g, h$ are as defined in (\ref{ghdef}). It follows that $\Lflat(\thefun) \subset T_{g, h}$, and so $\dim_H \Lflat(\thefun) \leq 1$. This completes the proof of Theorem~\ref{Tthird}.
\end{proof}
\begin{proof}[Proof of Theorem~\ref{Tmain}]
Choose $R \geq R_0$, where $R_0$ is as in the statement of Theorem~\ref{lbtheo}, and let $\underline{t}$ be as defined in (\ref{tdef}). Recall that $\dim_H I_{R}(\underline{t}) \geq 1$. We deduce Theorem~\ref{Tmain} from Theorem~\ref{Tthird} and the fact that $I_{R}(\underline{t}) \subset \Lunif(\thefun)  \subset \Lflat(\thefun).$
\end{proof}
\begin{proof}[Proof of Theorem~\ref{Tfourth}]
Choose $R \geq R_0$, where $R_0$ is as in the statement of Theorem~\ref{lbtheo}, and let $\underline{t'}$ be the sequence 
\begin{equation*}
\underline{t'} = 1 \ 4 \ 9 \ \ldots \ (n+1)^2 \ \ldots.
\end{equation*}
We deduce by Theorem~\ref{lbtheo} that $\dim_H I_{R}(\underline{t'}) \geq 1$. Suppose that $z \in I_{R}(\underline{t'})$, in which case $R^{(n+1)^2} \leq |\thefun^n(z)| \leq R^{(n+1)^2+1}$, for $n\geq 0$. It follows that $z \in\Mflat(\thefun)$. Hence $I_{R}(\underline{t'}) \subset \Mflat(\thefun)$, and so $\dim_H \Mflat(\thefun) \geq 1$. \\

Next, for each $p\isnatural$, let sequences $(g'_{p,n})_{n\isnatural}$ and $(h'_{p,n})_{n\isnatural}$ be defined by
\begin{equation}
\label{ghdashdef}
g'_{p,n}=e^{n\log^{+p}(n)} \text{ and } h'_{p,n} = \exp(e^{pn}), \qfor n\isnatural.
\end{equation}
We deduce by Theorem~\ref{ubtheo} that $\dim_H T_{g', h'}\leq 1$. 

Suppose that $z \in \Mflat(\thefun)$, in which case, by (\ref{Mflatdef}), there exist $p, N\isnatural$ and $R>1$ such that $$e^{n\log^{+p}(n)} \leq |\thefun^{n}(z)| \leq \exp(e^{pn}), \qfor n\geq N.$$ We deduce that $z \in T_{g', h'}$, where $g', h'$ are as defined in (\ref{ghdashdef}). It follows that $\Mflat(\thefun) \subset T_{g', h'}$, and so $\dim_H \Mflat(\thefun) \leq 1$. This completes the proof of Theorem~\ref{Tfourth}.
\end{proof}
Finally we prove Theorem~\ref{Tfirst}.
\begin{proof}[Proof of Theorem~\ref{Tfirst}]
Suppose that $f$ is a {\tef}, $R > 1$ and  $\underline{t} = t_0 t_1 t_2 \ldots$ is an escaping {\sai}. For each $p\isnatural$ we define sequences $(g_{p,n})_{n\isnatural}$ and $(h_{p,n})_{n\isnatural}$ by $g_{p,n} = R^{t_n}, \text{ and } h_{p,n} = R^{t_n+1}.$ We deduce by Theorem~\ref{ubtheo} that $\dim_H T_{g, h} \leq 1$. The first part of the theorem follows since $I_{R}(\underline{t}) \subset T_{g, h}$.

%
%

The second part of Theorem~\ref{Tfirst} is an immediate consequence of this, together with Theorem~\ref{lbtheo}.
\end{proof}
%
%
%
\section{The uniformly slowly escaping set}
\label{Sese}
In this section, for a general {\tef} $f$, we give two results on the set $\Lunif(f)$. First, we give a necessary and sufficient condition for $\Lunif(f)$ to be non-empty. Here $m(r, f) = \min_{|z| = r} |f(z)|$ denotes the \emph{minimum modulus} of $f$, for $r > 0$.
\begin{theorem}
\label{Texists}
Suppose that $f$ is a {\tef}. Then $$\Lunif(f)\cap J(f)\ne\emptyset$$ if and only if there exist positive constants $c$ and $r_0$, and $d > 1$ such that
\begin{equation}
\label{RSeq}
\text{for all } r \geq r_0 \text{ there exists } \rho \in (r, dr) \text{ such that } m(\rho, f) \leq c.
\end{equation}
Moreover, if $L_U(f) \ne \emptyset$, then $\Lunif(f)\cap J(f)\ne\emptyset$.
\end{theorem}
\begin{remark}
\normalfont The condition (\ref{RSeq}) was used in \cite{MR2792984}, also in relation to points tending to infinity at a specified rate. As observed in \cite{MR2792984}, this condition holds whenever $f$ is bounded on a path to infinity. It is well-known that this is the case for functions in the class $\mathcal{B}$ in particular.
\end{remark}
Second, we show that when $\Lunif(f)$ is not empty, it has a number of familiar properties which show that, in general, this is a dynamically interesting set. We say that a set $S$ is \emph{completely invariant} if $z \in S$ implies that $f(z) \in S$ and also that $f^{-1}(\{z\}) \subset S$.
\begin{theorem}
\label{Tprops}
Suppose that $f$ is a {\tef}, and that $\Lunif(f)~\ne~\emptyset$. Then the following all hold.
\begin{enumerate}[(i)]
\item $\Lunif(f)$ is completely invariant;
\item If $U$ is a Fatou component of $f$ and $U \cap \Lunif(f) \ne \emptyset$, then $U \subset \Lunif(f)$;
\item $\Lunif(f)$ is dense in $J(f)$ and $J(f) = \partial$$\Lunif(f)$.
\end{enumerate}
\end{theorem}
\begin{remark}\normalfont
An example of a {\tef} with a Fatou component contained in $\Lunif(f)$ is the function $f(z) = 2 - \log 2 + 2z - e^z,$ given by Bergweiler \cite{MR1317494}. It is straightforward to deduce from the arguments in \cite{MR1317494} that $f$ has a Baker domain in $\Lunif(f)$ and also wandering Fatou components in $\Lunif(f)$. We refer to \cite{MR1216719} for definitions.
\end{remark}

The proof of Theorem~\ref{Texists} requires the following \cite[Theorem 2]{MR2792984}.
\begin{lemma}
\label{LRS}
Suppose that $f$ is a {\tef}. Then $f$ has the property that, for all positive sequences $(a_n)_{n\isnatural}$ such that $a_n\rightarrow\infty$ as $n\rightarrow\infty$ and $a_{n+1} = O(M(a_n, f))$ as $n\rightarrow\infty$, there exist $\zeta\in J(f)$ and $C > 1$ such that $$a_n \leq |f^n(\zeta)| \leq Ca_n, \qfor n\isnatural,$$ if and only if there exist positive constants $c$ and $r_0$, and $d > 1$ such that (\ref{RSeq}) holds.
\end{lemma}
\begin{proof}[Proof of Theorem~\ref{Texists}]
Suppose that $f$ is a {\tef}. If there exist positive constants $c$ and $r_0$, and $d > 1$ such that (\ref{RSeq}) holds, then it follows immediately from Lemma~\ref{LRS} that $\Lunif(f) \cap J(f) \ne \emptyset$.

The other direction proceeds very similarly to the proof of the corresponding direction in \cite[Theorem 2]{MR2792984}. We give some details for completeness. Suppose that there do not exist positive constants $c$ and $r_0$, and $d > 1$ such that (\ref{RSeq}) holds. Then there exists a sequence of annuli $A(r_n,R_n)$, where $0 < r_n < R_n$, such that $r_n\rightarrow\infty$ as $n\rightarrow\infty$, $R_n/r_n\rightarrow\infty$ as $n\rightarrow\infty$ and $$m(r, f) > 1, \qfor r_n < r < R_n, \ n \isnatural.$$ As shown in the proof of \cite[Theorem 2]{MR2792984}, it follows that there exists $\delta \in (0,1)$ and $N\isnatural$ such that
\begin{equation}
\label{meq}
m(r, f)>M(r, f)^\delta, \qfor 2r_n < r < \frac{1}{2}R_n, \ n \geq N.
\end{equation}

We shall deduce that $\Lunif(f) = \emptyset$. Suppose, by way of contradiction, that there exists $z\in\Lunif(f)$. It follows from (\ref{LUdef}) that there exists $C>0$ and $N\isnatural$ such that
\begin{equation}
\label{zinLU}
f^n(z)\ne 0 \text{ and } \left|\frac{f^{n+1}(z)}{f^n(z)}\right| \leq C, \qfor n\geq N.
\end{equation}

We deduce that, for infinitely many values of $n\isnatural$, there exists $p(n)\isnatural$ such that $f^n(z) \in A(2r_{p(n)}, \frac{1}{2}R_{p(n)})$. It follows by (\ref{meq}) that $$\left|\frac{f^{n+1}(z)}{f^n(z)}\right| > \frac{M(|f^n(z)|,f)^\delta}{|f^n(z)|}, \qfor \text{infinitely many values of } n\isnatural.$$ Since, as is well-known, $$\frac{M(r, f)^\delta}{r}\rightarrow\infty \text{ as } r\rightarrow\infty,$$ this is a contradiction to (\ref{zinLU}). We deduce that $\Lunif(f) \cap J(f) = \emptyset$ if and only if (\ref{RSeq}) holds. Finally, if $\Lunif(f) \ne \emptyset$, then (\ref{RSeq}) holds, and so $\Lunif(f) \cap J(f) \ne \emptyset$.  This completes the proof of Theorem~\ref{Texists}.
\end{proof}
In order to prove Theorem~\ref{Tprops}, we require the following well-known distortion lemma; see, for example, \cite[Lemma 7]{MR1216719}.
\begin{lemma}
\label{dlemm}
Suppose that $f$ is a {\tef}, and that $U \subset I(f)$ is a simply connected Fatou component of $f$. Suppose that $K$ is a compact subset of $U$. Then there exist $C > 1$ and $N_0\isnatural$
such that
\begin{equation*}
\frac{1}{C}|f^n(z)| \leq|f^n(w)| \leq C|f^n(z)|, \qfor w, z \in K, n \geq N_0.
\end{equation*}
\end{lemma}

We also use the following, which is a special case of \cite[Lemma 10]{MR2792984}. We say that a set $S$ is \emph{backwards invariant} if $z \in S$ implies that $f^{-1}(\{z\}) \subset S$.
\begin{lemma}
\label{Jlemm}
Suppose that $f$ is a {\tef}, and that $E~\subset~\mathbb{C}$ contains at least three points. Suppose also that $E$ is backwards invariant under $f$, that $\operatorname{int} E\cap J(f) = \emptyset$, and that every component of $F(f)$ that meets $E$ is contained in $E$. Then $\partial E = J(f)$.
\end{lemma}
\begin{proof}[Proof of Theorem~\ref{Tprops}]
Suppose that $f$ is a {\tef}, and that $\Lunif(f)\ne\emptyset$, in which case, by Theorem~\ref{Texists}, $\Lunif(f)\cap J(f)\ne\emptyset$. First we observe that part (i) of the theorem follows immediately from the definition of $\Lunif(f)$.\\

For part (ii) of the theorem, suppose that $U$ is a Fatou component of $f$, such that $U \cap \Lunif(f) \ne \emptyset$. It follows by normality that $U \subset I(f)$.

Suppose that $U$ is multiply connected in which case \cite[Theorem 1.2]{Rippon01102012} we have that $U\subset A(f)$. However, it is known \cite{MR1684251} that if $z \in A(f)$ then $$\frac{\log \log |f^n(z)|}{n}\rightarrow\infty \text{ as } n\rightarrow\infty,$$ in which case $z \notin \Lunif(f)$. We deduce that $U$ is simply connected.

Suppose that $z \in U \cap \Lunif(f)$. Then there exist $N\isnatural, \ R > 1$ and $0 < C_1 < C_2$ such that $C_1R^{n} \leq |f^{n}(z)| \leq C_2R^{n}, \text{ for } n\geq N.$ Suppose that $K$ is a compact subset of $U$ containing $z$. Then, by Lemma~\ref{dlemm} there exist $C > 1$ and $N_0\isnatural$ such that $$\frac{C_1}{C} R^{n} \leq |f^{n}(w)| \leq C_2CR^{n}, \text{ for } w \in K, \ n\geq \max\{N, N_0\}.$$ Hence $K \subset \Lunif(f)$, and so $U \subset \Lunif(f)$. This completes the proof of part (ii).\\

For part (iii) of the theorem, we note that $\Lunif(f)\cap J(f)$ is an infinite set, since for each $z \in \Lunif(f)\cap J(f)$ at least one of the points $z, f(z)$ or $f^2(z)$ must have an infinite backwards orbit. It is known \cite[Theorem 4]{MR1216719} that the set of repelling periodic points of $f$ is dense in $J(f)$. Since, by definition, $\Lunif(f)$ contains no periodic points, $\operatorname{int} \Lunif(f) \subset F(f)$. The result follows by part (i) and part (ii), and by Lemma~\ref{Jlemm} applied with $E=\Lunif(f) \cap J(f)$ and then with $E=\Lunif(f)$. This completes the proof of Theorem~\ref{Tprops}.
\end{proof}
%
%
%
\section{Results of Karpi{\'n}ska and Urba{\'n}ski}
\label{SKU}
As mentioned in the introduction, Karpi{\'n}ska and Urba{\'n}ski \cite{MR2197375} studied the size of various subsets of $I(\thefun)$. For integers $k\geq 0$ and $l\geq k$, and $\epsilon > 0$, they defined sets $$D_\epsilon^{k,l} = \left\{ z \in I(\thefun): \Real(\thefun^n(z)) > q, \text{ for } n \geq k \text{, and }  |\operatorname{Im}(\thefun^n(z))| \leq \frac{|\thefun^n(z)|}{(\log |\thefun^n(z)|)^\epsilon}, \text{ for } n\geq l\right\},$$ where $q$ is fixed and large. Their main result is the following.
\begin{theorem}
For every $\epsilon > 0$ and all integers $0 \leq k \leq l,$ $$\dim_H D_\epsilon^{k,l} = 1 + \frac{1}{1 + \epsilon}.$$
\end{theorem}
In this section we show that the sets $D_\epsilon^{k,l}$ lie in the \emph{fast} escaping set of $\thefun$.
First, we define the domain $V = \{ z : \operatorname{Re}(z) > \frac{1}{2}|z|\}$. Note that
\begin{equation}
\label{ifinV}
|\thefun(z)| = |\lambda|e^{\operatorname{Re}(z)} > |\lambda|e^{\frac{1}{2}|z|} = M\left(\frac{1}{2}|z|, \thefun\right), \qfor z \in V.
\end{equation}
Suppose that $\epsilon > 0$ and $0 \leq k \leq l$, and that $z \in D_\epsilon^{k,l}$. It follows from the definition of $D_\epsilon^{k,l}$ that there exists $N\isnatural$ such that $\thefun^{n+N}(z) \in V$, for $n\geq 0$. Set $R = \frac{1}{2}|\thefun^N(z)|$, and define $$\mu(r) = \frac{1}{2}M(r, \thefun), \qfor r \geq 0.$$ We may assume that $R$ is sufficiently large that $\mu(r) > r$, for $r \geq R$. It follows from (\ref{ifinV}) that $$|\thefun^{n}(\thefun^N(z))| > \mu^n(R), \qfor n\geq 0,$$
and so, by \cite[Theorem 2.9]{Rippon01102012}, that $z \in A(\thefun)$, as required.\\

%
%
%
%
%
%
%
\emph{Acknowledgment:}
The author is grateful to Phil Rippon and Gwyneth Stallard for all their help with this paper. The author is also grateful to Chris Bishop for useful discussions.
\bibliographystyle{acm}
\bibliography{../../Research.References}
\end{document}